\documentclass[onefignum,onetabnum]{siamart220329}

\usepackage{lipsum}
\usepackage{amsfonts}
\usepackage{graphicx,cite}
\usepackage{epstopdf}

\usepackage{algorithm}
\usepackage{algorithmicx}
\usepackage{algpseudocode}

\ifpdf
  \DeclareGraphicsExtensions{.eps,.pdf,.png,.jpg}
\else
  \DeclareGraphicsExtensions{.eps}
\fi

\newsiamremark{remark}{Remark}
\newsiamremark{example}{Example}
\newsiamremark{hypothesis}{Hypothesis}
\crefname{hypothesis}{Hypothesis}{Hypotheses}
\newsiamthm{claim}{Claim}

\headers{GMRES, pseudospectra, and Crouzeix's conjecture}{T. Chen, A. Greenbaum, T. Trogdon}

\title{GMRES, pseudospectra, and Crouzeix's conjecture for shifted and scaled Ginibre matrices%
\thanks{%Submitted to the editors on June 17, 2021.
    \textbf{Funding:} This material is based on work supported by the National Science Foundation under Grant Nos. DGE-1762114 (TC), DMS-1945652 (TT).
Any opinions, findings, and conclusions or recommendations expressed in this material are those of the authors and do not necessarily reflect the views of the National Science Foundation.}}

\author{Tyler Chen\thanks{New York University, \href{mailto:tyler.chen@nyu.edu}{\texttt{tyler.chen@nyu.edu}}}
\and Anne Greenbaum\thanks{University of Washington, \href{mailto:greenbau_uw.edu}{\texttt{greenbau@uw.edu}}}
\and Thomas Trogdon\thanks{University of Washington, \href{mailto:trogdon@uw.edu}{\texttt{trogdon@uw.edu}}}
}

\usepackage{amsopn}

\makeatletter
\newcommand{\@nderle}[2]{%
  \vtop{
    \lineskiplimit\maxdimen
    \lineskip+.5\p@
    \ialign{$\m@th#1\hfil##\hfil$\crcr<\crcr#2\crcr}%
  }%
}
\newcommand{\lesim}{\mathrel{\mathpalette\@nderle\sim\relax}}

\newcommand{\@ndereq}[2]{%
  \vtop{
    \lineskiplimit\maxdimen
    \lineskip+.5\p@
    \ialign{$\m@th#1\hfil##\hfil$\crcr=\crcr#2\crcr}%
  }%
}
\makeatother
\crefname{section}{Section}{Sections}
\crefname{ineq}{}{}
\creflabelformat{ineq}{#2{\upshape(#1)}#3}

\usepackage{thmtools} 
\usepackage{thm-restate}

\usepackage{marginnote}
\usepackage{enumitem}

\usepackage{tikz}
\usetikzlibrary{calc}
\usepackage{graphicx}
\usepackage{subcaption}
\usepackage{booktabs}
\usepackage{array}

\usepackage{afterpage}
\usepackage{pdflscape}

\newcommand{\note}[1]{{\color{red}#1}}

\usepackage{amssymb}
\usepackage{bbm}
\usepackage{dsfont}
\newcommand{\bOne}{\ensuremath{\mathds{1}}}

\newcommand{\PP}{\mathbb{P}}

\renewcommand{\Re}{\operatorname{Re}}
\renewcommand{\Im}{\operatorname{Im}}

\renewcommand{\d}[1]{\ensuremath{\mathrm{d}#1}}

\renewcommand{\vec}{\mathbf}
\newcommand{\lmin}{\lambda_\textup{min}}

\newcommand{\T}{\textup{\textsf{T}}}
\newcommand{\cT}{*}
\newcommand{\asto}{\xrightarrow[N\to\infty]{\textup{a.s.}}}
\newcommand{\prto}{\xrightarrow[N\to\infty]{\textup{prob.}}}

\newcommand{\ii}{\mathrm{i}}

\usepackage[normalem]{ulem}

\ifpdf
\hypersetup{
  pdftitle={},
  pdfauthor={T. Chen, A. Greenbaum, T. Trogdon}
}
\fi

\begin{document}

\maketitle

\begin{abstract}
    We study the GMRES algorithm applied to linear systems of equations involving a scaled and shifted \( N\times N \) matrix whose entries are independent complex Gaussians. 
    When the right hand side of this linear system is independent of this random matrix, the \( N\to\infty \) behavior of the GMRES residual error can be determined exactly. 
    To handle cases where the right hand side depends on the random matrix, we study the pseudospectra and numerical range of Ginibre matrices and prove a restricted version of Crouzeix's conjecture.
\end{abstract}

% REQUIRED
\begin{keywords}
GMRES, 
random matrices,
numerical range, 
pseudospectrum, 
Crouzeix's conjecture
\end{keywords}

% REQUIRED
\begin{AMS}
%65F60, % Numerical computation of matrix exponential and similar matrix functions
%65F50, % Computational methods for sparse matrices
68Q25, % Analysis of algorithms and problem complexity
65F35, %Numerical computation of matrix norms, conditioning, scaling
15A60 %Norms of matrices, numerical range, applications of functional analysis to matrix theory
\end{AMS}

\section{Introduction}

Solving linear systems of equations \( \vec{A} \vec{x} = \vec{b} \) is one the most important tasks in the computational sciences, and Krylov subspace methods are among the most widely used algorithms for this task.
If \( \vec{A} \) is Hermitian, or more generally normal, then the (exact arithmetic) behavior of Krylov subspace methods is comparatively well understood.
On the other hand, when \( \vec{A} \) is non-normal, the behavior of Krylov subspace methods can be extremely colorful and remains an ongoing area of research \cite{carson_liesen_strakos_22}. 
In this paper, we analyze behavior of GMRES on systems involving a scaled and shifted \( N\times N \) complex Gaussian random matrix \( \vec{A}_N \) called a Ginibre matrix (see \cref{sec:GMRES_random}).
While such matrices are non-normal (their eigenvalue condition number grows linearly with the matrix size \( N \) \cite{chalker_mehlig_98}) one would be incorrect to assume this is a difficult problem for GMRES.
Thus, this paper provides yet an another example of the statement of Edelman and Rao that ``it is a mistake to link psychologically a random matrix with the intuitive notion of a `typical' matrix'' \cite{edelman_rao_05}.

Like other Krylov subspace methods, GMRES applied for \( k \) steps to the system \( \vec{A}\vec{x} = \vec{b} \) outputs an approximation \( p(\vec{A})\vec{b} \) to \( \vec{A}^{-1} \vec{b} \) where \( p \) is a degree \( k-1 \) polynomial. 
The residual \( \vec{r}^{(k)} \) of the $k$th GMRES iterate is characterized by the fact that it has minimal 2-norm among all approximations of this form.
That is,
\begin{align}
    \label{eqn:GMRES_def}
    \|\vec{r}^{(k)}\|
    = \min_{\substack{\deg(p)< k}} \|\vec{b} - \vec{A} p(\vec{A})\vec{b}\|
    = \min_{\substack{\deg(p)\leq k\\ p(0)=1}} \|p(\vec{A})\vec{b}\|.
\end{align}
The expression \cref{eqn:GMRES_def} depends on the right hand side vector \( \vec{b} \) and its relation to the matrix \( \vec{A} \).
Often, we would like to obtain bounds for \( \| \vec{r}^{(k)} \| \) that do not depend on \( \vec{b} \) in more than a trivial way.
Perhaps the simplest such bound is
\begin{equation*}
%    \label{eqn:GMRES_ideal}
    \frac{\|\vec{r}^{(k)}\|}{\| \vec{b} \|} \leq \min_{\substack{\deg(p)\leq k\\ p(0)=1}} \|p(\vec{A})\|,
\end{equation*}
where \( \| p(\vec{A}) \| \) denotes the induced operator 2-norm of the matrix \( p(\vec{A}) \). 
The polynomial attaining this bound is typically called the ideal GMRES polynomial, and such polynomials have been studied in a range of settings \cite{greenbaum_trefethen_94,faber_liesen_tichy_13}.

For normal matrices, bounding \( \| p(\vec{A}) \| \) simply amounts to bounding \( p \) on the eigenvalues of \( \vec{A} \).
When the eigenvalues of \( \vec{A} \) are real, i.e. if \( \vec{A} \) is Hermitian, then this is essentially a problem in classical approximation theory.
However, for non-normal matrices, the eigenvalues alone are not enough to determine the convergence of GMRES. 
In fact, among all \( N\times N \) matrices with \( N \) prescribed eigenvalues, there exists a matrix \( \vec{A} \) and right hand side \( \vec{b} \) such that the corresponding GMRES residual norms \( \| \vec{r}^{(k)} \| \) for \( k=0,1,\ldots, N-1 \) are equal to any desired non-increasing positive sequence \cite{greenbaum_ptak_strakos_96}.

Therefore, in order to relate the estimation of \( \| p(\vec{A}) \| \) to a problem in scalar approximation theory, one must consider some set other than eigenvalues.
An open set \( \Omega\subset \mathbb{C} \) is said to be a \( K \)-spectral set for \( \vec{A} \) if for all functions \( f \) analytic on \( \Omega \) and extending continuously to the boundary \( \partial \Omega \),
\begin{equation*}
    \| f(\vec{A}) \| \leq K \| f \|_{\Omega}.
\end{equation*}
Here \( \| f \|_\Omega := \sup_{z\in \Omega} |f(z)| \).
For convenience, we will denote by \( C(\Omega, \vec{A}) \) the smallest value \( K \) so that \( \Omega \) is a \( K \)-spectral set for \( \vec{A} \); i.e. 
\begin{equation*}
    C(\Omega, \vec{A}) := \sup \{ \| f(\vec{A}) \| / \| f \|_{\Omega} : \text{\(f\) analytic on \( \Omega \) and continuous on \(\partial \Omega\)} \}.
\end{equation*}
Thus, for any set \( \Omega \) and matrix \( \vec{A} \), 
\begin{align}
    \label{eqn:GMRES_Kss}
    \frac{\|\vec{r}^{(k)}\|}{\| \vec{b} \|} \leq C(\Omega, \vec{A})  \min_{\substack{\deg(p)\leq k\\ p(0)=1}} \| p \|_{\Omega}.
\end{align}
Also, note that the maximum modulus principle implies that
\begin{align}\label{eq:subset}
    C(\Omega',\vec A) \leq C(\Omega,\vec A),
\end{align}
if $\Omega \subset \Omega'$.

One standard choice for \( \Omega \) is the numerical range 
\begin{equation*}
        W(\vec{A}) := \{ \vec{v}^\cT \vec{A} \vec{v} : \| \vec{v} \| = 1 \} 
\end{equation*}
for which is conjectured that \( C(W(\vec{A}),\vec{A}) \leq 2 \) for any matrix \( \vec{A} \) \cite{crouzeix_07}.
Alternately, we might choose \( \Omega \) to be the \( \epsilon \)-pseudospectrum  
\begin{equation*}
    \Lambda_{\epsilon}(\vec{A}) := \{ z \in \mathbb{C} : \|R(z,\vec{A}) \| \geq \epsilon^{-1} \}
\end{equation*}
where \( R(z,\vec{A}) := (z\vec{I} - \vec{A})^{-1} \) is the resolvent.
For any \( \epsilon > 0 \), using the fact that \( \| R(z,\vec{A}) \| = \epsilon^{-1} \) for \( z\in \partial \Lambda_{\epsilon} \),  
\begin{align}
    \label{eqn:GMRES_PS}
    \| f(\vec{A}) \|
    &= \left\| \frac{1}{2\pi \ii}\int_{\partial \Lambda_{\epsilon}} f(z) R(z,\vec{A}) \d{z} \right\|
    \leq \frac{\operatorname{len}(\partial \Lambda_{\epsilon}(\vec{A}))}{2\pi \epsilon} \| f \|_{\Lambda_{\epsilon}}.
\end{align}
That is, \( \Lambda_{\epsilon} \) is a \(   \operatorname{len}(\partial \Lambda_{\epsilon}(\vec{A})) / (2\pi \epsilon) \)-spectral set for \( \vec{A} \).

Both pseudospectra and the numerical range are widely used in the study of non-normal matrices \cite{trefethen_embree_05}.
For normal matrices, \( \Lambda_{\epsilon}(\vec{A}) \) consists of the union of disks of radius \( \epsilon \) about the eigenvalues, but for non-normal matrices, \( \Lambda_{\epsilon}(\vec{A}) \) can be significantly larger.
Therefore, the size of \( \Lambda_{\epsilon}(\vec{A}) \) can be viewed as a measure of the non-normality of \( \vec{A} \).
Likewise, the numerical range of a normal matrix is simply the convex hull of the eigenvalues while the numerical range can be significantly larger for non-normal matrices.

In the remainder of this paper, we consider GMRES applied to a system of equations involving a scaled and shifted Ginibre matrix.
This paper is mostly expository, with the aim of further highlighting connections between several areas of mathematics including numerical linear algebra, matrix analysis, and random matrix theory.
It is clear there is great potential for cross-pollination of ideas between these fields, so we hope that this paper will serve to further enable this process. 
Thus, while many of the statements we make are unlikely to surprise the right expert, we believe the connections between disciplines are of interest to a broad audience.
Technical aspects of our proofs are largely outsourced to the literature, as our aim is to provide a high level perspective on the topics at hand. 

\section{Preliminaries}

In this section we outline some basic notation and well as review some standard definitions and results regarding the convergence of random variables.

\subsection{Notation}

We denote by \( \mathcal{D}(c,r) \) the closed disk of radius \( r \) centered at \( c \) and by \( \mathcal{A}(\tilde{r},r) \) the closed annulus with inner radius \( \tilde{r} \) and outer radius \( r \) centered at the origin.

The Hausdorff distance between sets \( X \) and \( Y \) is denoted \( d_H(X,Y) \) and defined as
\begin{equation*}
    d_H(X,Y) = \max\left\{ \sup_{x\in X} \inf_{y\in Y} |x-y|,~ \sup_{y\in Y} \inf_{x\in X} |x-y| \right\}.
\end{equation*}
In this paper, one of \( X \) or \( Y \) will always be a closed disk in which case we often use the following lemma:
\begin{lemma}
    \label{thm:hausdorff_disks}
Let \( S \subset \mathbb{C} \) be a closed subset of the complex plane and \( r > 0 \).
Suppose that for some \( \epsilon > 0 \), \( \mathcal{D}(0,r-\epsilon) \subseteq S \subseteq \mathcal{D}(0,r+\epsilon) \).
Then,
\begin{equation*}
    d_H(S,\mathcal{D}(0,r)) < \epsilon.
\end{equation*}
Further, if \( S \) is convex, then the converse is also true.
\end{lemma}

\subsection{Convergence of random variables}
The sequence \( X_1, X_2, \ldots  \) is said to converge to a random variable \( X \) in probability if,
for all \( \epsilon > 0 \),
\begin{equation*}
    \lim_{N\to\infty} \PP \Big[ | X_N - X| < \epsilon \Big] = 1.
\end{equation*}
Alternately, this sequence is said to converge to \( X \) almost surely if, 
\begin{equation*}
    \PP\Big[ \lim_{N\to\infty} X_N = X \Big] = 1.
\end{equation*}
We respectively denote convergence in probability and almost surely by
\begin{equation*}
    X_N \prto X
    \qquad\text{and}\qquad
    X_N \asto X.
\end{equation*}
Almost sure convergence is equivalent to the condition that, for all \( \epsilon > 0 \),
\begin{align}
\label{eqn:as_eps}
    \bOne[ |X_N - X| < \epsilon ] \asto 1.
\end{align}
Here \( \bOne[ \texttt{true} ] = 1 \) and \( \bOne[\texttt{false}] = 0 \).
If \( A_N \) and \( B_N \) are sequences of events, we also have
\begin{equation}
\label{eqn:as_cond}
    \bOne[A_N] \asto 1 
    \qquad\Longrightarrow\qquad
    \bOne[B_N] \asto 1
    ,\qquad \text{if }A_N\subset B_N.
\end{equation}

For \( c \in \mathbb{R} \) fixed, we write
\begin{equation*}
    X_N \lesim c 
    ~\text{almost surely as \( N\to\infty \)}
\end{equation*}
if, for all \( \epsilon > 0 \),
\begin{equation*}
    \bOne[X_N < c + \epsilon] \asto 1.
\end{equation*}

\section{GMRES on random systems}
\label{sec:GMRES_random}

Let \( \vec{G}_N = \frac{1}{\sqrt{2N}}(\vec{X}_N + \ii \vec{Y}_N) \)  where \( \vec{X}_N \) and \( \vec{Y}_N \) are independent \( N \times N \) matrices with independently and identically distributed (iid) real standard Gaussian entries.\footnote{While the size of the matrices we deal with will vary (since \( N \) is variable), we assume are we are working with a single probability space and have a semi-infinite array of random variables defined on this probability space.  Then an $N \times N$ matrix is formed by taking a principal subblock of this infinite array. }
The matrix \( \vec{G}_N \) has iid standard complex normal entries and is called a (complex) Ginibre matrix.
Our main aim is to analyze the residual norm \( \| \vec{r}_N^{(k)} \| \) corresponding to the GMRES algorithm applied for \( k \) steps to linear systems of the form \( (\vec{I} + \sigma \vec{G}_N ) \vec{x} = \vec{b} \), where \( \sigma \in (0,1) \) is some fixed constant.
Similar results can be expected to hold for non-Gaussian entries satisfying moment conditions; i.e. universality.

In \cref{sec:GMRES_lim},  we consider the case when \( \vec{b} \) is independent of \( \vec{A} \). 
In this case, we can describe an exact the distributional parameterization of the Arnoldi algorithm underlying GMRES. 
This allows us to characterize the residual norm:
\begin{restatable}{theorem}{GMRESlim}
    \label{thm:GMRES_lim_rate}
    Let \( \vec{G}_N \) be an \( N\times N \) complex Ginibre matrix.
    Then, for \( \sigma \in (0,1) \) and \( \vec{b}_N \) independent of \( \vec{G}_N \), the step \( k \) GMRES residual norm \( \| \vec{r}_N^{(k)} \| \) for the linear system \( (\vec{I} + \sigma \vec{G}_N) \vec{x} = \vec{b} \)
    satisfies
    \begin{equation*}
%    \label{eqn:GMRES_lim_rate}
        \frac{\| \vec{r}_N^{(k)} \|}{\| \vec{b}_N \|}
        \prto
        \left( \frac{1-\sigma^2}{1-\sigma^{2+2k}}\right)^{1/2} \sigma^k.
    \end{equation*}
\end{restatable}

If \( \vec{b} \) is allowed to depend on \( \vec{G}_N \), then this estimate cannot be expected to hold.
We therefore turn to \cref{eqn:GMRES_Kss} with the aim of characterizing some \( K \)-spectral sets of \( \vec{G}_N \) (and therefore \( \vec{I} + \sigma \vec{G}_N \)). 
In \cref{sec:pseudospectra} we use existing theory to characterize the \( \epsilon \)-pseudospectrum of \( \vec{G}_N \):
\begin{restatable}{theorem}{pseudospectra}
\label{thm:pseudospectra}
Let \( \vec{G}_N \) be a complex Ginibre matrix of size \( N\times N \).
Then, for any $\epsilon>0$,
\begin{equation*}
    d_H \big( \Lambda_{\epsilon}(\vec{G}_N),\mathcal{D}(0,\mathfrak{e}^{-1}(\epsilon^2)) \big) \asto 0.
\end{equation*}
\end{restatable}
\noindent 
Here \( \mathfrak{e}:[1,\infty) \to [0,\infty) \) is a deterministic increasing function with \( \mathfrak{e}(1) = 0 \) defined by 
\begin{align}
    \label{eqn:frake}
    \mathfrak{e}(z) := \frac{8 d^2-(9-8 d)^{3/2} - 36 d +27}{8(1-d)}, \qquad d:= 1-|z|^2
\end{align}
and $\mathfrak{e}^{-1}:[0,\infty) \to [1,\infty)$ is the inverse of $\mathfrak{e}$.

The numerical range of \( \vec{G}_N \) is also nearly circular for \( N \) large. 
Specifically, it is known \cite{collins_gawron_litvak_zyczkowski_14} that
\begin{equation*}
    d_H(W(\vec{G}_N), \mathcal{D}(0,\sqrt{2}))
    \asto 0.
\end{equation*}
In \cref{sec:crouzeix} we use this fact, in conjunction with existing theory on the growth of matrix functions on the numerical range, to show that the numerical range of a Ginibre matrix is a 2-spectral set almost surely as \( N\to\infty \). 
More precisely:
\begin{restatable}{theorem}{crouzeix}
\label{thm:crouzeix}
Let \( \vec{G}_N \) be a complex Ginibre matrix of size \( N\times N \).
Then, 
\begin{equation*}
    C(W(\vec{G}_N), \vec{G}_N) \lesim 2
\end{equation*}
almost surely as \( N\to\infty \).
\end{restatable}
\noindent 

\begin{figure}
    \centering
    \includegraphics[scale=.8]{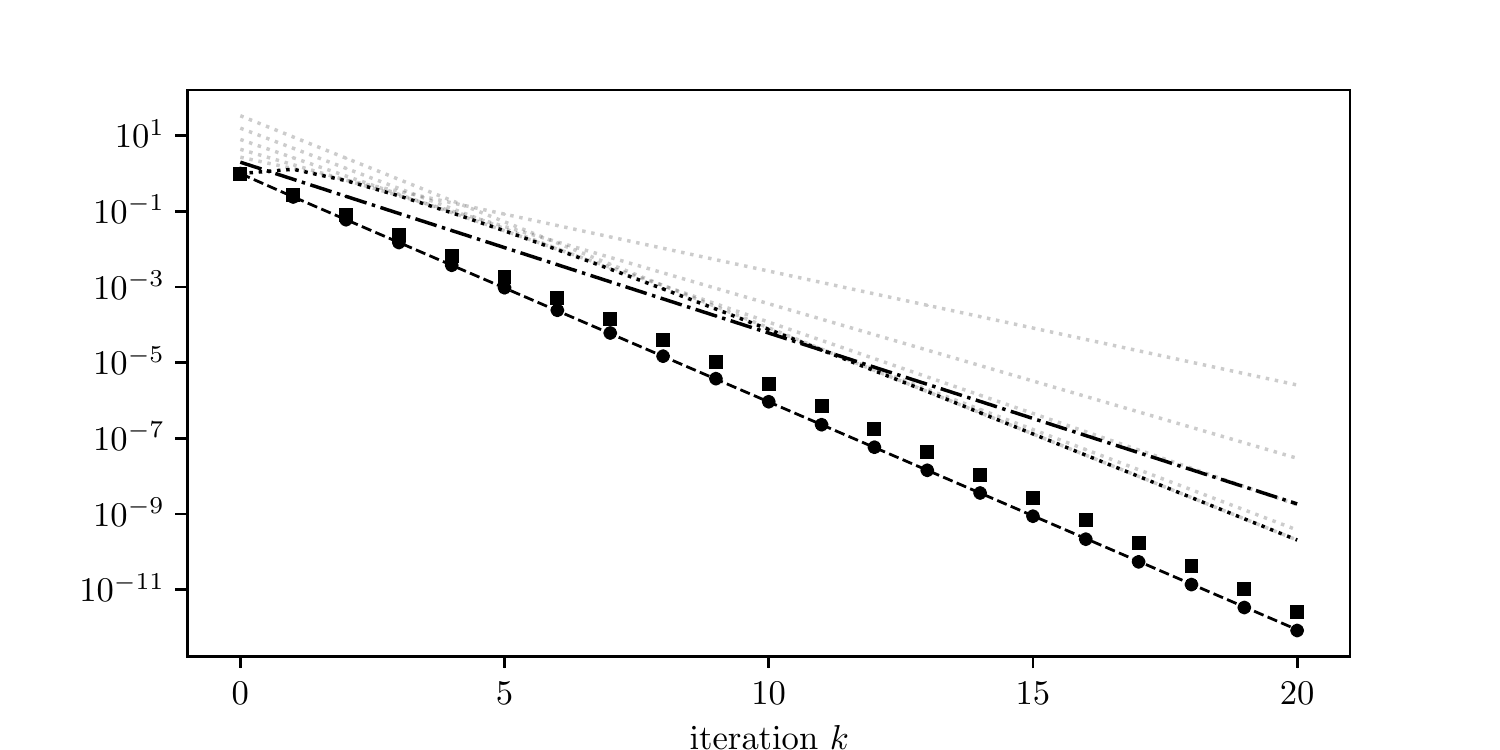}
    \caption{%
        Convergence of GMRES on linear systems \( (\vec{I} + \sigma \vec{G}_N) \vec{x} = \vec{b} \) for \( \vec{b} \) chosen independently and dependently of \( \vec{G}_N \). 
        When  $\vec{b}$ is chosen independently of $\vec{G}_N$, then the residual norm $\|\vec{r}_k\|$ of GMRES (circles) is described by \cref{thm:GMRES_lim_rate} (dashed line).
        If instead $\vec{b}$ is allowed to depend on $\vec{G}_N$, then the residual norm $\|\vec{r}_k\|$  of GMRES (squares) need not be described by \cref{thm:GMRES_lim_rate}. 
        However, the bounds in \cref{thm:GMRES_ideal_bound} based on the pseudospectrum (dotted lines) and numerical range (dash-dot line) are still relevant.
        In this experiment, $N=1500$ and $\sigma=1/4$.
    }
    \label{fig:GMRES_example}
\end{figure}

Since both the numerical range and pseudospectra are nearly circular for large \( N \), we may hope to obtain a bound in terms of \( \min \{ \| p \|_{\mathcal{D}(1,r)} : \deg(p) \leq k,~ p(0)=1 \} \).
This is an elementary problem which can be solved exactly. 
Indeed, on the disk \( \mathcal{D}(1,r) \), the minimal infinity norm degree \( k \) monic polynomial is \( (z-1)^{k} \), so we obtain our desired polynomial by rescaling the \( (z-1)^k \) to take value 1 at the origin; i.e. the minimizing polynomial is \( p(x) = (z-1)^{k}/(-1)^{k} = (1-z)^k\).
This gives the relation
\begin{equation}
    \label{eqn:minimax}
    \min_{\substack{\deg(p)\leq k\\p(0)=1}}  \| p \|_{\mathcal{D}(1,r)}
    =  r^{k}.
\end{equation}
Combining this with the above results we can prove the following bound.
\begin{restatable}{theorem}{GMRESidealbound}
    \label{thm:GMRES_ideal_bound}
    Let \( \vec{G}_N \) be an \( N\times N \) complex Ginibre matrix.
    Then, for \( \sigma \in (0,1) \) and \( \vec{b} \) possibly depending on \( \vec{G}_N \), the step \(  k \) GMRES residual norm \( \| \vec{r}_N^{(k)} \| \) for the linear system \( (\vec{I} + \sigma \vec{G}_N) \vec{x} = \vec{b} \) satisfies,
\begin{equation*}
    \frac{\| \vec{r}^{(k)} \|}{\| \vec{b} \|} \lesim  \min_{\eta>1}  \eta \mathfrak{e}(\eta)^{-1/2} (\sigma \eta)^{k}
    \qquad \text{and} \qquad
    \frac{\| \vec{r}^{(k)} \|}{\| \vec{b} \|} \lesim  2 (\sqrt{2}\sigma)^k 
\end{equation*}
almost surely as \( N\to\infty \).

\end{restatable}

\Cref{fig:GMRES_example} shows a numerical example with \( N = 1500 \) and \( \sigma = 1/4 \).
As expected, when \( \vec{b} \) is independent of \( \vec{G}_N \), the convergence is determined by \cref{thm:GMRES_lim_rate}.
Once we have sampled the random matrix \( \vec{I} + \sigma \vec{G}_N \), we then attempt to find a \( \vec{b} \) which increases the value of \( \| \vec{r}^{(k)} \| \).
Maximizing the residual norm at a given iteration is a hard problem \cite{faber_liesen_tichy_13},  so instead of seeking to find the worst case \( \vec{b} \) for each \( k \), we simply compute the \( \vec{b} \) which maximizes \( \| p(\vec{I} + \sigma \vec{G}_N) \vec{b} \|/\|\vec{b}\|  \) when \( p = (z-1)^{10} \).
This appears to be sufficient to break agreement with the rate in \cref{thm:GMRES_lim_rate}. 
Even so, the bounds in \cref{thm:GMRES_ideal_bound} remain valid.
% and the rate cannot be slowed significantly beyond \( \sigma^k \).
%We remark that it is unclear what worst case residuals are, this is a quantity which exhibits any limiting behavior.

\section{Limiting behavior of GMRES with independent right hand sides}
\label{sec:GMRES_lim}

We begin by studying the case where \( \vec{b} \) is independent of \( \vec{G}_N \). 
It has been observed that many algorithms have essentially deterministic behavior when applied to random matrices of large dimension \cite{pfrang_deift_menon_13,deift_menon_olver_trogdon_14}.
This has been rigorously established for a range of iterative methods for linear systems including 
conjugate gradient and MINRES \cite{deift_trogdon_20,paquette_trogdon_20,ding_trogdon_21} and 
gradient and stochastic gradient descent \cite{paquette_lee_pedregosa_paquette_21,paquette_paquette_21,paquette_merrienboer_paquette_pedregosa_22}.
Some numerical experiments for the GMRES algorithm in the case \( \sigma=1 \) are given in \cite{deift_menon_olver_trogdon_14} where one sees non-deterministic behavior (see also \cite{zhang_trogdon_22}).

It is well known, and can be seen from the joint density of the matrix entries, that Ginibre matrices are invariant under unitary conjugation. 
That is, for any fixed unitary matrix \( \vec{Q} \), we have that \( \vec{Q} \vec{G}_N \vec{Q}^\cT \stackrel{\text{dist.}}{=} \vec{G}_N \). 
Thus, without loss of generality, we can assume that \( \vec{b}/ \| \vec{b} \| = \vec{e}_1 := [1, 0, \ldots, 0]^\T \) by applying a unitary transform \( \vec{Q} \) with \( \vec{Q} \vec{b} = \| \vec{b} \| \vec{e}_1 \) to \( \vec{b} \) and considering the system involving \( \vec{I} + \sigma \vec{Q} \vec{G}_N \vec{Q}^\cT \) and \( \| \vec{b} \| \vec{e}_1 \).

We now construct a unitary matrix \( \vec{V} \) so that \( \vec{V} (\vec{I} + \sigma \vec{G}_N) \vec{V}^\cT\) is upper-Hessenberg and \( \vec{V} \vec{e}_1 = \vec{e}_1 \).
The residual norm obtained by GMRES applied for \( k \) iterations to the system \( \vec{A} \vec{x} = \vec{b} \) is the same as GMRES applied for \( k \) iterations to \( \vec{V} \vec{A} \vec{V}^{\cT} \vec{x} = \vec{V} \vec{b} \) for any unitary matrix \( \vec{V} \) (even if \( \vec{V} \) depends on \( \vec{A} \)). 
Indeed, using the unitary invariance of the 2-norm, 
\begin{equation*}
    \| p(\vec{A})  \vec{b} \|
    = \| p( \vec{A} ) \vec{V}^\cT \vec{V} \vec{b} \|
    = \| \vec{V} p( \vec{A} ) \vec{V}^\cT \vec{V} \vec{b} \|
    = \| p( \vec{V}  \vec{A} \vec{V}^\cT) \vec{V} \vec{b} \|.
\end{equation*}

The approach is by constructing suitable Householder reflectors, and is similar in spirit to approaches for other matrix ensembles \cite{trotter_84,silverstein_85,dumitriu_edelman_02}.
For convenience, partition \( \vec{G}_N \) as 
\begin{equation*}
    \vec{G}_N = 
    \frac{1}{\sqrt{2N}} \begin{bmatrix}
        g  & \vec{y}^\T  \\
        \vec{x} & \vec{G} \\
    \end{bmatrix}.
\end{equation*}
Now, conditioning on the probability one event that \( \|\vec{v}\| \neq 0 \), 
define \( \vec{U} \) as the Householder reflector
\begin{equation*}
    \vec{U}
    = \begin{bmatrix}
        1 & \vec{0}^\T \\
        \vec{0} & \vec{F} \\
    \end{bmatrix}
    ,\qquad 
    \vec{F} = \vec{I} - 2 \frac{\vec{v}\vec{v}^\cT}{\vec{v}^\cT\vec{v}}
    ,\qquad
    \vec{v} = \| \vec{x} \| \vec{e}_1  - \vec{x}.
\end{equation*}
Then \( \vec{U} \vec{e}_1 = \vec{e}_1 \) and 
\begin{equation*}
    \vec{U} \vec{G}_N \vec{U}^\cT
    =
    \frac{1}{\sqrt{2N}}
    \begin{bmatrix}
        g  & \vec{y}^\T \vec{F}^\cT \\
        \vec{F}\vec{x} & \vec{F}\vec{G} \vec{F}^\cT \\
    \end{bmatrix}.
\end{equation*}
By construction, \( \vec{F} \vec{x}  = [ \| \vec{x} \| , 0 , \ldots , 0 ]^\T \).
Moreover, since the real and imaginary parts of the entires of \( \vec{x} \) are iid standard normal random variables, 
\begin{equation*}
    \| \vec{x} \| \stackrel{\textup{dist.}}{=} \chi(2(N-1))
\end{equation*}
where \( \chi(j) \) is the Chi distribution with \( j \) degrees of freedom.
Next, note that \( \vec{y} \) and \( \vec{G} \) are independent of \( \vec{x} \) and therefore independent of \( \vec{F} \).
Then, by the unitary invariance of Gaussian vectors \( \vec{y}^\T \vec{F}^\cT \stackrel{\text{dist.}}{=} \vec{y}^\T \) and by the invariance of Gaussian matrices under unitary conjugation \( \vec{F} \vec{G} \vec{F}^\cT \stackrel{\text{dist.}}{=} \vec{G}\).

We can now apply this process to the submatrix \( \vec{F} \vec{G} \vec{F}^\cT \). 
Thus, inductively, we obtain (with probability one) a unitary matrix \( \vec{V} \) so that \( \vec{V} \vec{e}_1= \vec{e}_1 \) and
\begin{align}
    \label{eqn:hess}
    \vec{V}\vec{G}_N \vec{V}^\cT \stackrel{\text{dist.}}{=} 
\frac{1}{\sqrt{2N}}
\begin{bmatrix}
    \mathcal{N}_{\mathbb{C}}(0,2) & \cdots  & \cdots & \cdots & \mathcal{N}_{\mathbb{C}}(0,2) \\
    \chi(2(N-1)) & \ddots &  && \vdots \\
    & \chi(2(N-2)) & \ddots  && \vdots\\
    && \ddots & \ddots & \vdots \\
    &&&\chi(2(1)) & \mathcal{N}_{\mathbb{C}}(0,2)
\end{bmatrix}.
\end{align}
Here \( \mathcal{N}_{\mathbb{C}}(0,2) \) is a complex Gaussian with mean zero and variance 2.

Denote by \( \vec{H}_N \) the matrix in \cref{eqn:hess} and let \( \vec{A}_N := \vec{I} + \sigma \vec{H}_N \) for some \( \sigma \in (0,1) \).
Then the GMRES residual for the system \(  \vec{A}_N \vec{x} =\| \vec{b} \| \vec{e}_1 \) at step \( k \) has identical distribution to the GMRES residuals for the system \( (\vec{I} + \sigma \vec{G}_N) \vec{x}  = \vec{b} \) at step \( k \).

The residual norm \( \| \vec{r}^{(k)} \| \) for the system \( \vec{A} \vec{x} = \vec{b} \) can be written as 
\begin{align}
    \label{eqn:res_formula}
    \frac{\| \vec{r}^{(k)} \|^2}{\|\vec{b}\|^2}
    = 
    \frac{1}{1 + \| ([\widetilde{\vec{A}}]_{2:k+1,:k})^{-\cT} ([\widetilde{\vec{A}}]_{1,:})^\cT \|^2}
\end{align}
where \( \widetilde{\vec{A}} \) is the \( (k+1)\times k \) upper-Hessenberg matrix produced by the Arnoldi algorithm run for \( k \) iterations on \( \vec{A} \) and \( \vec{b} \) and $-*$ denotes the inverse conjugate transpose \cite[Theorem 5.1]{meurant_11}.
It's well known that the Arnoldi algorithm applied to a upper-Hessenberg matrix and the first unit vector will produce back the same upper-Hessenberg matrix. 
Thus, since \( \vec{A}_N \) is upper-Hessenberg, \( \widetilde{\vec{A}} = [\vec{A}_N]_{:k+1,:k} \).

With \( k \) remaining fixed, we will use \cref{eqn:res_formula} to analyze the GMRES residual norm in the \( N\to\infty \) limit. 
\iffalse
We have that
\begin{align}
[\vec{H}_N]_{:k+1,:k}
    \stackrel{\text{dist.}}{=}
\frac{1}{\sqrt{2N}}
\begin{bmatrix}
    \mathcal{CN}(0,2) & \cdots & \mathcal{CN}(0,2) \\
    \chi_{2(N-1)} & \ddots &  \vdots \\
    & \ddots & \mathcal{CN}(0,1) \\
    \qquad\chi_{2(N-k)} & 
\end{bmatrix}
\end{align}
so it follows easily that, 
\fi
Direct computation shows that
\begin{equation*}
    [\vec{A}_N]_{:k+1,:k}
    \prto
    \vec{I} + \sigma
    \begin{bmatrix}
        0 & \cdots & 0 \\
        1 & \ddots &  \vdots \\
        & \ddots & 0 \\
        & & 1  
    \end{bmatrix}
    =
    \begin{bmatrix}
        1 &  \\
        \sigma & \ddots \\
        & \ddots & 1 \\
        & & \sigma  
    \end{bmatrix}
    =: \vec{A}_\infty.
\end{equation*}
Now note that \( [\vec{A}_\infty]_{1,:} = \vec{e}_1 \) so that, by basic properties of Jordan blocks,
\begin{equation*}
    ([\vec{A}_\infty]_{2:k+1,:k})^{-\cT} \vec{e}_1
    = \begin{bmatrix}
        \sigma &  \\
        1 & \ddots & \\
        & \ddots & \ddots &\\
        & &  1&\sigma  
    \end{bmatrix}^{-1} \vec{e}_1
    = 
    \begin{bmatrix}
        \sigma^{-1} \\
        -\sigma^{-2} \\
        \sigma^{-3} \\
        \vdots \\
        (-1)^{k-1}\sigma^{-k}
    \end{bmatrix}.
\end{equation*}
We therefore have
\begin{equation*}
     1 + \| ([\widetilde{\vec{A}}_\infty]_{2:k+1,:k})^{-\cT} ([\widetilde{\vec{A}}_\infty]_{1,:})^\T \|^2 
    = 1+\sum_{i=1}^{k} \sigma^{-2i}
    = \left( \frac{1-\sigma^{2+2k}}{1-\sigma^2} \right) \sigma^{-2k}.
\end{equation*}
Thus, using \cref{eqn:res_formula} and the continuity of the matrix inverse in the neighborhood of any invertible matrix, we obtain:
\GMRESlim*
We remark that it would be interesting to study the fluctuations in \( \| \vec{r}^{(k)} \| \), either by characterizing the asymptotic distribution or deriving quantitative bounds for the rate of convergence in probability.
This has been done for the related conjugate gradient and MINRES algorithms \cite{deift_trogdon_20,paquette_trogdon_20,ding_trogdon_21}.

\section{Resolvent norms and pseudospectra of Ginibre matrices}
\label{sec:pseudospectra}

The analysis in the previous section relied on the fact that \( \vec{b} \) is independent of \( \vec{G}_N \).
If \( \vec{b} \) is allowed to depend on \( \vec{G}_N \), then the estimate in \cref{thm:GMRES_lim_rate} need not hold.
Indeed, in \cref{fig:GMRES_example} we illustrated an example where the estimate is far from accurate.
As such, we turn to \cref{eqn:GMRES_Kss}.
We begin by studying the pseudospectra of Ginibre matrices by bounding resolvent norms.
For some nice pictures and a high level discussion of pseudospectra of random matrices, see \cite[Chapter 35]{trefethen_embree_05}.

For \( r,\tilde{r} \) with \( r \geq \tilde{r} \geq 1 \), recall that \( \mathcal{A}(\tilde{r},r) \) is the closed annulus with inner radius \( \tilde{r} \) and outer radius \( r \). 
Our first goal is to show the resolvent is bounded on \( \mathcal{A}(\tilde{r},r) \):
\begin{lemma}
\label{thm:resolvent_bound}
Let \( \vec{G}_N \) be an \( N\times N \) complex Ginibre matrix.
Then for any  \( r,\tilde{r} \) with \( r \geq \tilde{r} \geq 1 \), for all \( \epsilon > 0 \), 
\begin{equation*}
    \bOne \Big[ \forall z\in\mathcal{A}(\tilde{r},r): ~ \mathfrak{e}(r)^{-1/2} - \epsilon < \| R(z,\vec{G}_N) \| < \mathfrak{e}(\tilde{r})^{-1/2} + \epsilon \Big]
    \asto 1.
\end{equation*}   
\end{lemma}
Towards this end, note that, for any \( z \),
\begin{equation*}
    \| R(z,\vec{G}_N) \|  = \lmin( \vec{Y}_N^z )^{-1/2}
    ,\qquad\text{where} \qquad
    \vec{Y}_N^z := ( z \vec{I} - \vec{G}_N)^\cT ( z \vec{I} - \vec{G}_N).
\end{equation*}
Studying the eigenvalues of \( \vec{Y}_N^z \) is a common approach for studying the resolvent norm since \( \vec{Y}_N^z \) is Hermitian and therefore potentially simpler to analyze. 

Let \( \rho_N^z(t) \) be the empirical spectral measure of \( \vec{Y}_N^z \); i.e. 
\begin{equation*}
    \rho_N^z(t) := \frac{1}{N} \sum_{i=1}^{N} \delta(t-\lambda_i(\vec{Y}_N^z))
\end{equation*}
where \( \delta(t) \) is the delta distribution centered at zero and \( \{ \lambda_i(\vec{Y}_N^z) \}_{i=1}^{N} \) are the eigenvalues of \( \vec{Y}_N^z \).
In the \( N\to\infty \) limit, \( \rho_N^z(t) \) converges in distribution to a deterministic limiting density \( \rho^z(t) \) almost surely \cite{dozier_silverstein_07}.
Specifically, \cite[Theorem 1.1]{dozier_silverstein_07} shows that the associated Stieltjes transform
\begin{equation*}
    m^z(w) := \int_{-\infty}^{\infty} \frac{\rho^z(t)}{w-t}\d t
\end{equation*}
satisfies a certain integral equation determined by properties of the information and noise matrices.
For shifted Ginibre matrices, the equation for the Stieltjes transform reduces to an algebraic relation
\begin{equation*}
    \frac{1}{m^z(w)} - \frac{1}{1+|z|^2 m^z(w)} + (1+|z|^2 m^z(w)) w = 0,
\end{equation*}
where it is required that \( \Im( m^z(w) )> 0 \) if \( \Im(w) > 0 \) \cite[Equation 9]{cipolloni_erdos_schroder_20}.
From this expression, the support of \( \rho^z(t) \) can be directly computed.
In particular, for \( z \) with \( |z|>1 \), as seen in \cite[Equation 18a]{cipolloni_erdos_schroder_20}, \( \rho^z(t) \) is supported on \( [\mathfrak{e}(z), \mathfrak{f}(z)] \) where 
\begin{align}
    \label{eqn:Y_support}
    \mathfrak{e}(z), \mathfrak{f}(z)  := \frac{8 d^2\pm (9-8 d)^{3/2} - 36 d +27}{8(1-d)}, \qquad d:= 1-|z|^2.
\end{align}
For \( |z|>1 \), the limiting density \( \rho^z(t) \) has square root behavior at the edges \( \mathfrak{e}(z) \), \( \mathfrak{f}(z) \) \cite[Equation 18b]{cipolloni_erdos_schroder_20}.
This means that \( \rho^z(t) \) is non-negative just to the right of \( \mathfrak{e}(z) \) so that, for all \( \epsilon > 0 \), \( \bOne[ \lmin(\vec{Y}_N^z ) <  \mathfrak{e}(z) + \epsilon ] \to 1 \) almost surely as \( N\to\infty \). 

It is known that almost surely no eigenvalues of information plus noise matrices lie outside the support of the limiting spectral density \cite{vallet_loubaton_mestre_12,bai_silverstein_12}.
In particular, for any \( \epsilon > 0 \),
\begin{equation*}
    \bOne \Big[  \lmin(\vec{Y}_N^z ) > \mathfrak{e}(z) -\epsilon \Big] 
    \asto 1.
\end{equation*}
%Note that \cite[Theorem 1]{bai_silverstein_12} in fact applies to the more general class of information-plus-noise type matrices, and can be viewed as a generalization of the classical result that there are no eigenvalues outside the support for sample covariance matrices \cite{bai_silverstein_98}.
Thus, combining the previous results and using that $\mathfrak{e}(z)$ is increasing as a function of $|z|$, for each \( z \in \mathcal{A}(r, \tilde r)\), for all \( \epsilon > 0 \), 
\begin{align}
    \label{eqn:resolvent_bound_single}
    \bOne \Big[ \mathfrak{e}(r)^{-1/2} - \epsilon < \| R(z,\vec{G}_N) \| < \mathfrak{e}(\tilde{r})^{-1/2} + \epsilon \Big] 
    \asto 1.
\end{align}
The relationship between \( \| R(z,\vec{G}_N) \| \) and \( |z| \) is explored numerically in \cref{fig:pseudospectra}.

\begin{figure}
    \centering
    \includegraphics[scale=.8]{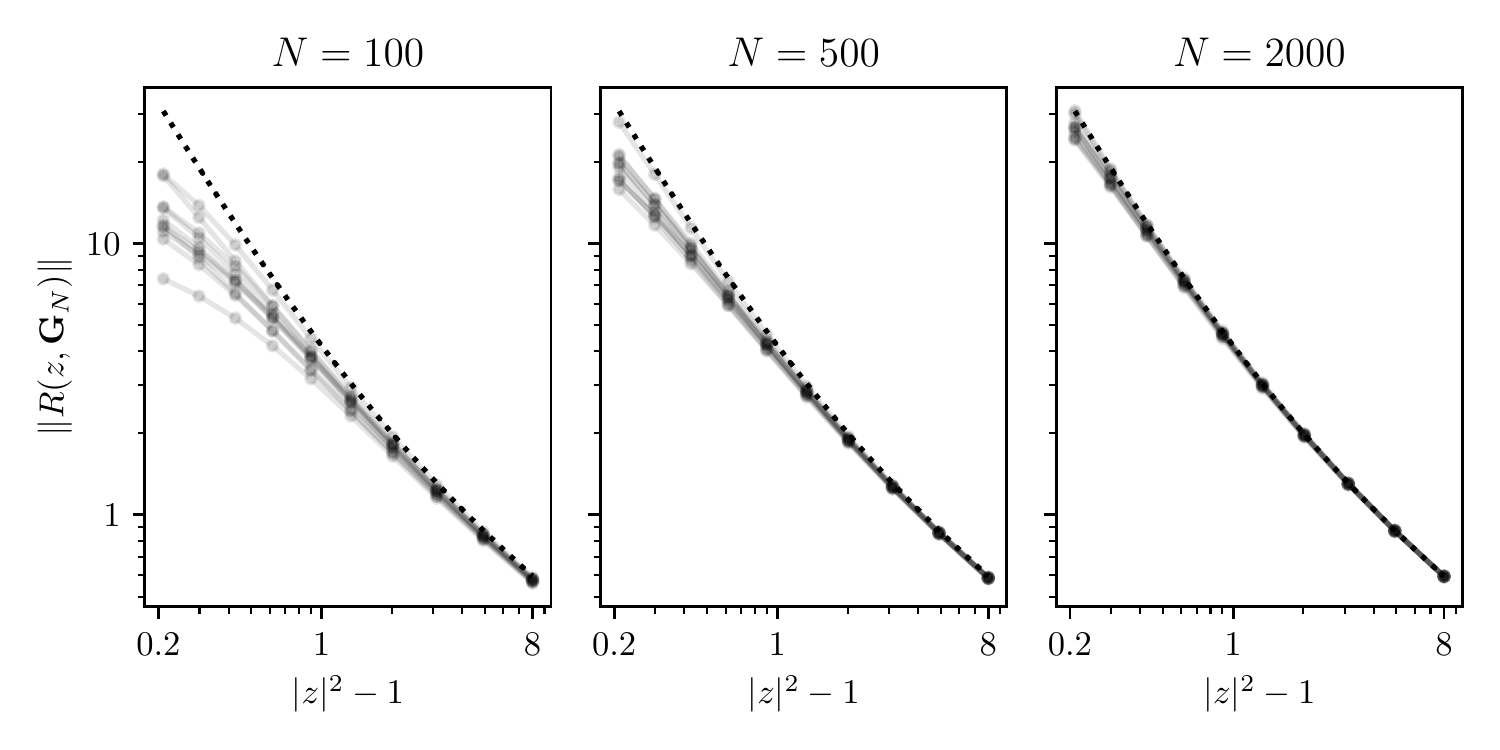}
    \caption{Relationship between \( \| R(z,\vec{G}_N) \| \) and \( |z|^2-1 \) for 10 samples of \( \vec{G}_N \). 
    The dotted curve is \( \mathfrak{e}(z)^{1/2} \). 
    Here we take \( z \) on the real axis, although the same result is expected for any \( z \) due to the rotational symmetry of \( \vec{G}_N \).
    Note that as \( N\to\infty \), the estimate \( \mathfrak{e}(z)^{1/2} \) becomes more accurate.}
    \label{fig:pseudospectra}
\end{figure}

We will now upgrade \cref{eqn:resolvent_bound_single} to simultaneously hold for all \( z \in \mathcal{A}(r, \tilde r)\); i.e. prove \cref{thm:resolvent_bound}.
Our basic approach is to construct a finite set \( \mathcal{N}_{\epsilon}(r,\tilde r) \subset \mathcal{A}(r,\tilde r) \) such that for any point \( \tilde{z}\in\mathcal{A}(r, \tilde{r}) \),  there is a point \( z \in \mathcal{N}_{\epsilon}(\tilde{r},r) \) for which \( \| R(z,\vec{G}_N) \| \) is bounded and for which \( |z-\tilde{z}| \) is small enough so that \( \| R(\tilde{z},\vec{G}_N) \| \) is close to \( \| R(z,\vec{G}_N) \| \).
Towards this end, set \( e = \mathfrak{e}(r)^{-1/2} + 0.1 \), \( f = r + 0.1 \), and \( g = 2.1 \).
Then, for each \( z\in \mathcal{A}(r, \tilde r)\), 
\begin{align}
    \label{eqn:lip_cond_single}
    \bOne\Big[ \| R(z,\vec{G}_N) \| < e,~ |z| < f,~ \| \vec{G}_N \| < g \Big]
    \asto 1.
\end{align}
Here we have used the results above and the well-known fact that \( \| \vec{G}_N \| \lesim 2 \) almost surely as $N\to\infty$ \cite{geman_80}.

Fix \( \epsilon > 0 \) and, with the benefit of hindsight, define \( h := \min\{1, (4e^2(f+g+1))^{-1} \} \) and \( L := 3 e^{3} (f+g+1)  \).
Let \( \mathcal{N}_{\epsilon}(r,\tilde r) \subset \mathcal{A}(\tilde{r},r) \) be a finite set of points  so that any other point in \( \mathcal{A}(r, \tilde r) \) is within \( \min\{ h/2, \epsilon/(2L) \}  \) of a point in \( \mathcal{N}_{\epsilon}(r,\tilde r) \).
This is possible since \( \mathcal{A}(r, \tilde r) \) is compact and \( \min\{h/2,\epsilon/(2L)\} > 0 \).
Then, since \( \mathcal{N}_{\epsilon}(r,\tilde r) \) contains a finite number of points, 
\cref{eqn:resolvent_bound_single} implies
\begin{align}
    \label{eqn:resolvent_bound_e3}
    \bOne \Big[ \forall z\in\mathcal{N}_{\epsilon}(r,\tilde r): ~ \mathfrak{e}(\tilde{r})^{-1/2} - \frac{\epsilon}{2} < \| R(z,\vec{G}_N) \| < \mathfrak{e}(r)^{-1/2} + \frac{\epsilon}{2} \Big]
    \asto 1
\end{align}
and \cref{eqn:lip_cond_single} implies
\begin{align}
    \label{eqn:resolvent_condition}
    \bOne\Big[ \forall z\in\mathcal{N}_{\epsilon}(r,\tilde r): \| R(z,\vec{G}_N) \| < e,~ |z| < f,~ \| \vec{G}_N \| < g \Big]
    \asto 1.
\end{align}
Let \( \tilde z\in \mathcal{A}(r,\tilde r) \).
By construction, there exists \( z\in\mathcal{N}_{\epsilon}(r,\tilde r) \) such that \( |z-\tilde z| \leq \min\{ h/2 , \epsilon/(2L) \} \).
Condition on the event: \( \{ \forall z\in\mathcal{N}_{\epsilon}(r,\tilde r):~\| R(z,\vec{G}_N) \| < e,~ |z| < f,~ \| \vec{G}_N \| < g \}  \). 
Then, reveling in our excellent choices of \( h \) and \( L \), we can use basic properties of the resolvent norm (see \cref{sec:resolvent_lip}) to show that
\begin{equation*}
    \big| \| R(z,\vec{G}_N) \| - \| R( \tilde z,\vec{G}_N) \| \big| \leq L |z-\tilde{z}| < \epsilon/2
\end{equation*}
which, combined with \cref{eqn:resolvent_bound_e3}, implies that
\begin{equation*}
    \mathfrak{e}(\tilde{r})^{-1/2} + \epsilon 
    <  \| R(\tilde z,\vec{G}_N) \| < 
    \mathfrak{e}(r )^{-1/2} + \epsilon. 
\end{equation*}
Applying \cref{eqn:as_cond} yields \cref{thm:resolvent_bound}.

\subsection{Pseudospectra}

\Cref{thm:resolvent_bound} gives a bound for the \( \epsilon \)-pseudospectrum of \( \vec{A} \).
Recall $\mathfrak{e}:[1,\infty)\to [0,\infty)$ is increasing with $\mathfrak{e}(1) = 0$, and relate $\epsilon>0$ with $r>1$ by $\epsilon = \mathfrak{e}(r)^{-1/2}$.
Then, for all $\epsilon>0$ (\( r > 1 \)) and for all \( \tilde{\epsilon} > 0 \), using \cref{thm:resolvent_bound} and the fact that \( \bOne[ \forall z \in \mathcal{D}(0,1) : \| R(z,\vec{G}_N) \| > T ] \asto 1 \) for any \( T>0 \), we have 
\begin{align}
    \label{eqn:resolvent_disk}
    \bOne \Big[ \forall z\in \mathcal{D}(0,r) :  \| R(z,\vec{G}_N) \| > \mathfrak{e}(r)^{-1/2} - \tilde{\epsilon} \Big]
    &\asto 1
    \\
    \bOne \Big[ \forall z\in \mathcal{A}(r,\infty) :  \| R(z,\vec{G}_N) \| < \mathfrak{e}(r)^{-1/2} + \tilde{\epsilon} \Big]
    & \asto 1.
\end{align}

We will use this to show that the \( \epsilon \)-pseudospectrum of \( \vec{G}_N \) is near to \( \mathcal{D}(0,r) \) when \( N \) is large. 
Fix $\tilde{\epsilon}>0$ and condition on the events: \( \{ \forall z\in \mathcal{D}(0,r): \| R(z,\vec{G}_N) \| > \epsilon  -\tilde{ \epsilon}  \} \) and \( \{ \forall z\in \mathcal{D}(r,\infty): \| R(z,\vec{G}_N) \| < \epsilon +\tilde{ \epsilon} \} \).
Then
\begin{equation*}
    \mathcal{D}(0,r -\tilde{ \epsilon}) \subset \Lambda_\epsilon(\vec{G}_N) \subset \mathcal{D}(0,r+\tilde{\epsilon}).
\end{equation*}
Thus, by \cref{thm:hausdorff_disks} we have
\begin{equation*}
    d_H(\Lambda_{\epsilon}(\vec{G}_N), \mathcal{D}(0,r)) < \tilde{\epsilon}.
\end{equation*}
Using \cref{eqn:as_cond} we therefore establish:
\pseudospectra*

\iffalse

First, condition on the event: \( \forall z\in\mathcal{D}(0,r) \), \( \|R(z,\vec{G}_N) \| < \eta + \epsilon \).
This \note{NOO} implies \( \mathcal{D}(0,r) \subset \Lambda_{\eta} \oplus \mathcal{D}(0,\epsilon) \) so that, \cref{eqn:resolvent_disk} implies
\begin{equation*}
    \bOne \Big[ \mathcal{D}(0,r) \subset \Lambda_{\eta} \oplus \mathcal{D}(0,\epsilon) \Big] 
    \asto 1
\end{equation*}
where ``\:\( \oplus \)\:'' indicates the Minkowski addition of sets.

Now, condition non the event: \( \forall z\in\mathcal{D}(0,r) \), \(\)

\note{more careful here}

show:

That there cannot be any other components follows from the circular law. 
Specifically by \cite[Theorem 2]{rider_03} we have that,  
\begin{equation*}
    \bOne\Big[ \text{all eigenvalues of \( \vec{G}_N \) in \( \mathcal{D}(0,\tfrac{1}{2}(\eta+1) \)} \Big] 
    \asto 1.
\end{equation*}
Each connected component of the pseudospectrum must contain at least one eigenvalue.

Thus,
\begin{equation*}
    \bOne \Big[ \Lambda_{\eta} \subset  \mathcal{D}(0,r)\oplus \mathcal{D}(0,\epsilon) \Big] 
    \asto 1.
\end{equation*}
\fi

Note that this result is for the fixed \( \epsilon \)-pseudospectrum (i.e. \( |z|>1 \) fixed).
From a random matrix theory perspective, it is more interesting to study other limits. 
Bounds on the eigenvalues of \( \vec{Y}_N^z \) in the case \( |z|<1 \) are used in proofs of the circular law, 
and the case \( |z| \leq 1 + CN^{1/2} \) is the main focus of \cite{cipolloni_erdos_schroder_20}.
Other work focuses on pseudospectra directly. 
For instance, \cite{bourgade_dubach_19} studies the volume of \( \Lambda_{\epsilon} \) in an \( \epsilon \to 0 \) limit where \( \epsilon \) depends on \( N \).

\section{The numerical range of Ginibre matrices and Crouzeix's conjecture}
\label{sec:crouzeix}

We now turn our attention to the numerical range.
For any matrix \(   \vec{A} \), the numerical range \( W(\vec{A}) \) is convex. 
It is not hard to show that \(  \operatorname{Re}(W(\vec{A})) =  W(\operatorname{He}(\vec{A})) \), where \( \operatorname{He}(\vec{A}) := (\vec{A} + \vec{A}^{\cT})/2 \) is the Hermitian part of $\vec{A}$.
The real part of the rightmost point of \( W(\operatorname{He}(\vec{A})) \) is simply the largest eigenvalue of \( \operatorname{He}(\vec{A}) \).
Thus, we can obtain a tangent line to the numerical range: $\lambda_{\textup{max}} ( \operatorname{He} (\vec{A})) + t \ii$, $t \in ( - \infty , \infty )$. 
Now, note that \( \exp(\ii\theta) W(\vec{A}) =  W(\exp(\ii\theta)\vec{A}) \); i.e. the numerical range is preserved under rotations of the complex plane. 
Thus, applying the above procedure to \( \exp(\ii\theta) \vec{A} \), rather than \( \vec{A} \), allows us to compute a set of tangent lines for the numerical range. 
Since the numerical range is convex, this procedure will construct a polygon enclosing the numerical range, which converges to the numerical range as more values of \( \theta \) are evaluated.
This is a standard technique for computing the boundary of the numerical range (see for instance \cite{johnson_78}) which we use to plot the numerical range for several instances of \( \vec{G}_N \) in \cref{fig:numerical_range}.

This technique can be applied to Ginibre matrices analytically \cite{collins_gawron_litvak_zyczkowski_14}.
Towards this end, note that
\begin{equation*}
    \operatorname{He}(\vec{G}_N) = \frac{\vec{G}_N + \vec{G}_N^{\cT}}{2}
\end{equation*}
is a scaled Gaussian Unitary Ensemble (GUE) matrix. 
The eigenvalues of GUE matrices are well studied, and when scaled so that the diagonal entries have variance \( 1/N \), the limiting distribution is a semicircle on \( [-2,2] \).
Indeed, this is the celebrated Wigner semicircle law \cite{wigner_58} which arguably pioneered what is now called random matrix theory.
The diagonal entries of \( \operatorname{He}(\vec{G}_N) \) have variance \( 1/(2N) \), so accounting for this difference, we therefore have that \cite{bai_yin_88}
\begin{align}
    \label{eqn:GUE_lim}
    \|\! \operatorname{He}(\vec{G}_N) \| \asto \sqrt{2}.
\end{align}

This fact is then applied to rotated matrices \( \exp(\ii\theta) \vec{G}_N \), \( \theta\in[0,2\pi) \).
Since \( \exp(\ii\theta) \vec{G}_N \stackrel{\text{dist.}}{=} \vec{G}_N \), the analog of \cref{eqn:GUE_lim} applies for all fixed \( \theta \). 
A covering argument \cite[Theorems 4.1]{collins_gawron_litvak_zyczkowski_14} similar to the one we used above allows the result to be transferred from fixed \( \theta \) to simultaneously hold for all \( \theta \in [0,2\pi) \). 
The result is that, as expected, the numerical range of \( \vec{G}_N \) converges to the disk of radius \( \sqrt{2} \) centered at the origin almost surely as \( N\to\infty \) \cite[Theorem 4.1]{collins_gawron_litvak_zyczkowski_14}.
Specifically, 
\begin{align}
    \label{eqn:nr_disk}
    d_H(W(\vec{G}_N), \mathcal{D}(0,\sqrt{2})) 
    \asto 0.
\end{align}

\begin{figure}
    \centering
    \includegraphics[scale=.8]{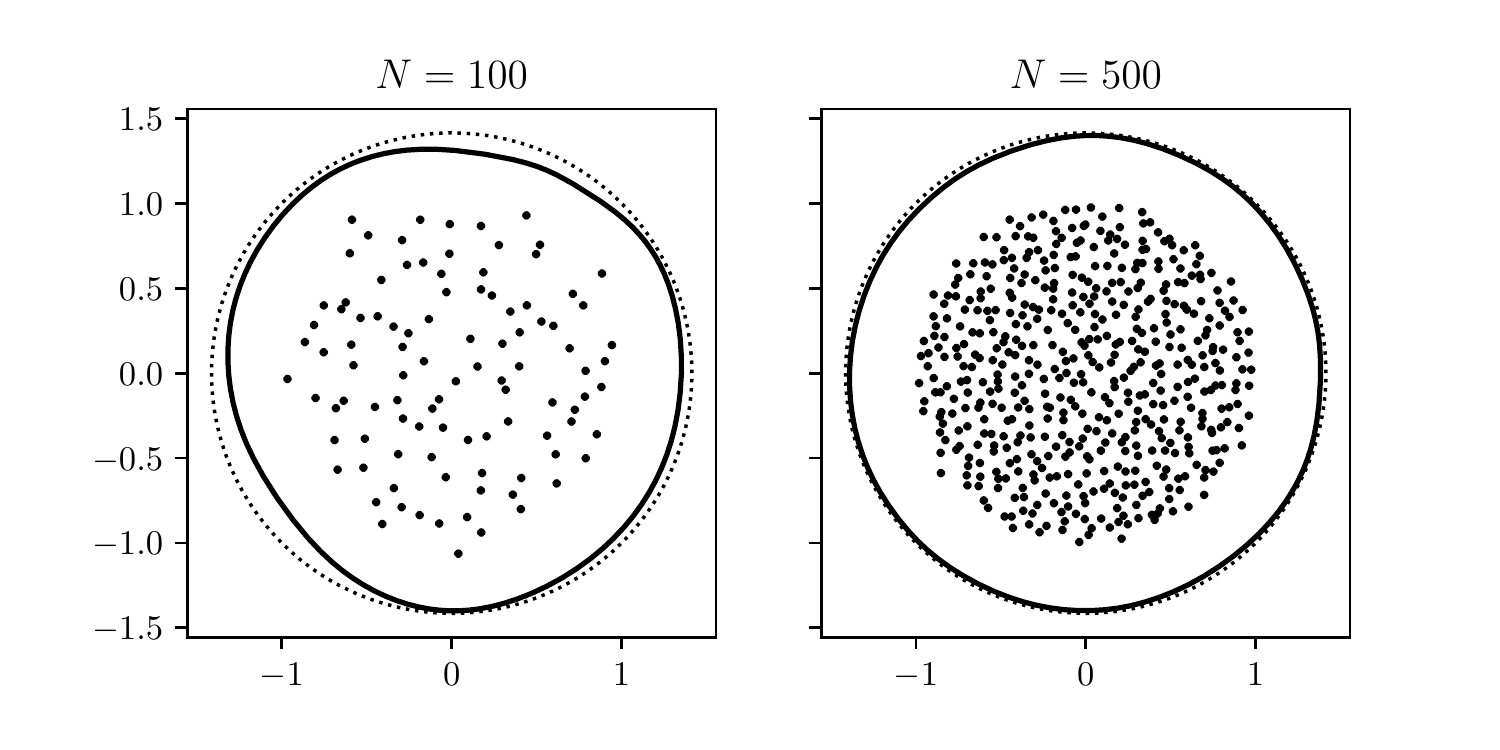}
    \caption{Numerical range (interior of solid curve) and eigenvalues (dots) of \( \vec{G}_N \). 
    The boundary of the disk \( \mathcal{D}(0,\sqrt{2}) \) is shown as the dotted circle. Note that as \( N\to\infty \), the numerical range approaches \( \mathcal{D}(0,\sqrt{2}) \).}
    \label{fig:numerical_range}
\end{figure}

%\note{lower bound 2?}
An open question in matrix analysis is determining the minimum value $K$ so that   \( W(\vec{A}) \) is a $K$-spectral set for every matrix $\vec{A}$. 
Crouzeix's conjecture is that this minimum value is 2 \cite{crouzeix_04,bickel_gorkin_greenbaum_ransford_schwenninger_wegert_20}.
It is known that the numerical range is a \( (1+\sqrt{2}) \)-spectral set for any \( \vec{A} \) \cite{crouzeix_palencia_17}, and for practical purposes, this bound is hardly worse than the conjectured value of \( 2 \). 
Even so, determining classes of matrices for which the value \( (1+\sqrt{2}) \) can be improved remains an active area of research.
In this section, we add to these results by establishing a version of Crouzeix's conjecture for large Ginibre matrices.

If the numerical range of a matrix is circular, then the numerical range is a \( 2 \)-spectral set for the given matrix \cite{okubo_ando_75,caldwell_greenbaum_li_18}.
We will show a perturbative version of this statement which will hold for the numerical range of \( \vec{G}_N \); i.e. that nearly circular numerical range are nearly \( 2 \)-spectral sets.
Our exposition follows that of \cite{caldwell_greenbaum_li_18} which itself is based on \cite{crouzeix_palencia_17}.
It is an interesting question whether the numerical range of \( W(\vec{G}_N) \) is nearly a \( K \) spectral set for \( \vec{G}_N \) for any \( K < 2 \).
We have performed numerical experiments which suggest the numerical range of $\vec{G}_N$ may nearly be a \( \sqrt{2} \)-spectral set.
This is discussed further in \cref{sec:nr_blaschke}.

Let \( \Omega \) be a region with with a smooth boundary containing in its interior the spectrum of a matrix \( \vec{A} \). 
Then for any function \( f \) analytic on \( \Omega \) and continuous on the boundary,
\begin{equation*}
    f(\vec{A}) 
    = \frac{1}{2\pi \ii} \int_{\partial \Omega} f(\sigma) R(\sigma,\vec{A}) \d\sigma,
\end{equation*}
where \( R(\sigma,\vec{A}) := (\sigma \vec{I} - \vec{A})^{-1} \) is the resolvent.
Now define
\begin{equation*}
    g(\vec{A}) := \frac{1}{2\pi \ii} \int_{\partial \Omega} \overline{f(\sigma)} R(\sigma,\vec{A}) \d\sigma
\end{equation*}
and consider the matrix
\begin{equation*}
    f(\vec{A}) + g(\vec{A})^{\cT}  = \int_0^L \mu(s,\vec{A}) f(\sigma(s)) \d s
\end{equation*}
where
\begin{equation*}
    \mu(s,\vec{A}) :=  \left[ \frac{\sigma'(s)}{2\pi \ii} R(\sigma(s),\vec{A}) \right] + \left[ \frac{\sigma'(s)}{2\pi \ii} R(\sigma(s),\vec{A}) \right]^\cT.
\end{equation*}
Here we have parametrized \( \partial \Omega \) by \( s \in [0,L] \) with \( \sigma : [0,L] \to \partial \Omega \) giving the map from the parametrization to the boundary.

When \( \Omega = W(\vec{A}) \), it can be shown that \( \| f(\vec{A}) + g(\vec{A})^\cT \| \leq 2 \) from which the constant \( K = 1+\sqrt{2} \) is then obtained.  
One of the aims of \cite{caldwell_greenbaum_li_18} was to consider \( \Omega \) other than the numerical range, for instance a set just inside the numerical range. 
In this case, \( \mu(s,\vec{A}) \) may not be positive semidefinite so the authors introduce the necessarily positive semidefinite matrix
\begin{equation*}
    \mu(s,\vec{A}) - \lmin(\mu(s,\vec{A})) \vec{I}.
\end{equation*}
From this, they establish \cite[Lemma 2.1]{caldwell_greenbaum_li_18} that
\begin{align}
    \label{eqn:fggam_bd}
    \| f(\vec{A}) + g(\vec{A})^\cT + \gamma \vec{I} \| \leq 2 + \delta
\end{align}
where
\begin{equation*}
    \gamma := -\int_0^L \lmin(\mu(s,\vec{A})) f(\sigma(s)) \d s
    \qquad\text{and}\qquad
    \delta := -\int_0^L \lmin(\mu(s,\vec{A})) \d s.
\end{equation*}
Note that if \( \sigma \) is positively oriented, then \( \sigma'(s) / \ii \) is the outward normal vector to \( \partial \Omega \) at \( \sigma(s) \). 
Therefore, assuming \( \Omega \) is convex, if we view \( \mu \) as a scalar function, then the set
\begin{equation*}
    \{ z \in \mathbb{C} \setminus \{ \sigma(s) \} : \mu(s,z) > 0 \}
    = \{ z \in \mathbb{C} \setminus \{ \sigma(s) \} : \Re ((\sigma'(s)/\ii)/(\sigma(s)-z) > 0 \}
\end{equation*}
is the open half plane containing \( \Omega \) which is tangent to \( \partial \Omega \) at \( \sigma(s) \).
Next, note that
\begin{equation*}
    \lmin(\mu(s,\vec{A}))
    = \min_{\|\vec{v}\|=1} \vec{v}^\cT \mu(s,\vec{A}) \vec{v}
    = \min_{\|\vec{v}\|=1} \Re \Big[ \frac{\sigma'(s)}{\pi \ii} \vec{v}^\cT  R(\sigma(s), \vec{A}) \vec{v} \Big].
\end{equation*}
Thus, we see that the sign of \( \delta \) depends on \( \Omega \) in relation to the numerical range \cite{caldwell_greenbaum_li_18}.
Specifically, if \( \Omega = W(\vec{A}) \) then \( \delta = 0 \), 
if \( \Omega \subsetneq \operatorname{int}(W(\vec{A}))  \) then \( \delta > 0 \), 
and if \( \Omega \supsetneq \operatorname{cl}(W(\vec{A})) \) then \( \delta < 0 \).

\subsection{Nearly circular numerical range}

For any any \( N\times N \) matrix \( \vec{A} \) there exists a function \( f \) which attains the ratio \( C(W(\vec{A}),\vec{A}) \).
Without loss of generality, we will assume $\|f\|_{W(\vec{A})} = 1$.
It is known that this function is of the form \( f = B \circ \phi \) where \( B \) is a finite Blaschke product of degree at most \( N-1 \) and \( \phi \) is any conformal mapping from the field of values to the unit disk. 
Suppose \( B \) is the Blaschke product which maximizes \( \| (B\circ \phi)(\vec{A}) \| \) among all Blaschke products of degree at most \( N-1 \).
Then, assuming \( \| (B\circ\phi)(\vec{A}) \| > 1  \), \( \vec{u}_1^\cT \vec{v}_1 = 0 \), where \( \vec{u}_1 \) and \( \vec{v}_1 \) are the left and right singular vectors of \( (B\circ\phi)(\vec{A}) \) corresponding to the largest singular value.

Let \( \Omega \) be any disk \( \mathcal{D}(c,r) \) containing the numerical range.
For \(  z\in \Omega \), provided \( f \) is analytic in a neighborhood of \(   \Omega \), 
it's not hard to show that
\begin{equation*}
    g(z) 
    = \frac{1}{2\pi \ii}\int_{\partial \Omega} \overline{f(\sigma)}(\sigma-z)^{-1} \d\sigma
    = \overline{f(c)}.
\end{equation*}
Let \( f = B \circ \phi \) where \( \phi(z) = z/r \) and \( B \) maximizes \( \| (B\circ\phi)(\vec{A}) \| \).
Then if \( \| (B\circ\phi)(\vec{A}) \| > 1 \),
\begin{equation*}
    \vec{u}_1^\cT( f(\vec{A}) + g(\vec{A})^{*} + \gamma \vec{I} ) \vec{v}_1 = \vec{u}_1^\cT f(\vec{A}) \vec{v}_1 = \| f(\vec{A}) \|.
\end{equation*}
This, in conjunction with \cref{eqn:fggam_bd} and the fact $W(\vec{A})\subseteq \Omega$ shows that 
\begin{equation*}
    \| f(\vec{A}) \| \leq \max\{ 1, 2 + \delta \} \leq 2.
\end{equation*}

Now, let \( \tilde{\Omega} \) be any disk \( \mathcal{D}(c,\tilde{r}) \) contained in the numerical range.
Define $\sigma (s) := c + r \exp ( \ii s/r )$ and $\tilde{\sigma} ( \tilde{s} ) := c + \tilde{r} \exp ( \ii \tilde{s} / \tilde{r} )$ for $s \in [0, 2 \pi r ]$ and $\tilde{s} \in [0, 2 \pi \tilde{r} ]$.
We will take $\tilde{s} = (\tilde{r}/r)s$.
Then 
\begin{equation*}
    | \sigma(s) - \tilde{\sigma}(\tilde{s}) | = r - \tilde{r}
    \qquad \text{and} \qquad
    \sigma'(s) = \tilde{\sigma}'(\tilde{s}) = \ii \exp(\ii s/r).
\end{equation*}
Define \( \tilde\mu \) and \( \tilde\delta \) analagously to \( \mu \) and \( \delta \).
Then the above argument shows that, as long as the eigenvalues of \( \vec{A} \) are contained within \( \tilde\Omega \), then \( \tilde\Omega \) is a \( \max\{1,2+\tilde\delta \}\)-spectral set for \( \vec{A} \). 
Our aim is to bound \( |\tilde\delta| \) in terms of a quantity depending on \( r- \tilde{r} \).

Note that \(  \lmin(\mu(s,\vec{A})) \) is positive whereas \( \lmin(\tilde{\mu}(\tilde{s},\vec{A})) \) is negative.
Thus, using standard eigenvalue perturbation bounds we find
\begin{equation*}
    | \lmin(\tilde{\mu}(\tilde{s},\vec{A}))|
    \leq
    | \lmin(\mu(s,\vec{A})) - \lmin(\tilde{\mu}(\tilde{s},\vec{A})) | 
    \leq \| \mu(s,\vec{A}) - \tilde{\mu}(\tilde{s},\vec{A}) \|.
\end{equation*}
Now, by definition, 
\begin{equation*}
    \mu(s,\vec{A}) - \tilde{\mu}(\tilde{s},\vec{A}) 
    =
    2 \operatorname{He}\left[ \frac{\sigma'(s)}{2\pi \ii} R(\sigma(s),\vec{A}) - \frac{\tilde\sigma'(\tilde{s})}{2\pi \ii}R(\tilde{\sigma}(\tilde{s}),\vec{A} ) \right].
\end{equation*}
Thus, using the first resolvent identity, and suppresing the dependence of \( \sigma(s) \) and \( \tilde\sigma(\tilde{s}) \) on \( s \) and $\tilde{s}$ for notational clarity,
\begin{align*}
    \frac{\sigma'}{2\pi \ii} R(\sigma,\vec{A}) - \frac{\tilde\sigma'}{2\pi \ii}R(\tilde{\sigma},\vec{A} )
    &=
    \frac{\sigma'}{2\pi \ii} R(\sigma,\vec{A}) - \frac{\sigma'}{2\pi \ii}R(\tilde{\sigma},\vec{A} ) + \frac{\sigma'-\tilde\sigma'}{2\pi \ii} R(\tilde\sigma,\vec{A})
    \\&= - \frac{\sigma'}{2\pi \ii} (\sigma -\tilde\sigma) R(\sigma,\vec{A}) R(\tilde\sigma,\vec{A}) + 0.
\end{align*}
Thus, using the fact that $\|\! \operatorname{He}(\vec{B}) \| \leq \|\vec{B}\|$ and $|\sigma'(s)| = 1$,
\begin{equation*}
 | \lmin(\tilde{\mu}( \tilde{s},\vec{A}))|
    \leq \frac{r-\tilde{r}}{\pi} 
      \| R(\sigma,\vec{A}) \| \| R(\tilde\sigma,\vec{A}) \| ,
\end{equation*}
and therefore, provided that \( \| R(z,\vec{A}) \| \leq \Delta \) for all \( z \in\partial{\Omega} \cup \partial\tilde\Omega \),
\begin{align}
    \label{eqn:delta_int_bound}
    |\tilde\delta| \leq \int_0^{2\pi} | \lmin(\tilde{\mu}(s,\vec{A}))| \d s
    \leq 2 ( r - \tilde{r} ) \Delta^2.
\end{align}
Assuming that \( |r - \tilde{r}| \) is small relative to \( \Delta^2 \), this implies that  \( \tilde\Omega \), and therefore \( W(\vec{A}) \), is a \( (2+\tilde\delta ) \)-spectral set for some \( |\tilde\delta| \) small.

In the case of Ginibre matrices, for any \( \eta \in (0,\sqrt{2}-1.1) \), we can take \( r = \sqrt{2} + \eta \) and \( \tilde{r} = \sqrt{2} - \eta \). 
From \cref{eqn:nr_disk,thm:hausdorff_disks} we see that
that, for all \( \eta \in (0,\sqrt{2}) \),
\begin{equation*}
    \bOne \left[ \tilde\Omega \subset W(\vec{G}_N) \subset \Omega \right]
    \asto 1.
\end{equation*}
Moreover, all eigenvalues of \( \vec{G}_N \) are contained in the disk \( \mathcal{D}(0,1.1) \subset \tilde\Omega \) almost surely as \( N\to\infty \). 
Since \( |r-\tilde{r}| = 2\eta \), from \cref{thm:resolvent_bound}, as long as $\eta<\sqrt{2} - 1.1$, it suffices to take \( \Delta = \mathfrak{e}(1.1)^{-1/2} \).
Thus, for any \( \epsilon > 0 \), \( \tilde\Omega \) is a \( (2+\epsilon) \)-spectral set for \( \vec{G}_N \) almost surely as \(  N\to\infty \). 
The maximum modulus principle then implies that:
\crouzeix*

\section{Deferred proofs}

\subsection{Proof of main bound}

\GMRESidealbound*
\begin{proof}
    
    In both cases we will use the fact that if \( \Omega \) is a \( K \)-spectral set for \( \vec{G}_N \) then \( 1 + \sigma \Omega \) is a \( K \)-spectral set for \( \vec{I} + \sigma\vec{G}_N \).

    Fix \( \eta > 1 \).
 The eigenvalues of \( \vec{G}_N \) are contained in \( \mathcal{D}(0,(\eta+1)/2) \subset \mathcal{D}(0,\eta) \) almost surely as \( N\to\infty \) \cite{tao_11} (this is also implied by \cref{thm:pseudospectra} with $\epsilon$ sufficiently small).
    Conditioning on the eigenvalues of \( \vec{G}_N \) being contained in \( \mathcal{D}(0,\eta) \), 
    \begin{equation*}
        \| f(\vec{G}_N) \|
        = \left \| \frac{1}{2\pi \ii} \int_{\partial \mathcal{D}(0,\eta)} f(z) R(z,\vec{G}_N) \d{z} \right\|
        \leq \frac{2\pi \eta}{2\pi} \| f \|_{\mathcal{D}(0,\eta)} \sup_{z\in\partial\mathcal{D}(0,\eta)} \| R(z,\vec{G}_N) \|.
    \end{equation*}
    \Cref{thm:pseudospectra} implies that, for any \( \epsilon > 0 \), 
    \begin{equation*}
        \bOne \Big[ \sup_{z\in\partial\mathcal{D}(0,\eta)} \| R(z,\vec{G}_N) \| < \mathfrak{e}(\eta)^{-1/2} + \epsilon \Big]
        \asto 1.
    \end{equation*}
    Therefore 
    \begin{equation*}
        \bOne \Big[
            \mathcal{D}(0,\eta) \text{ is a \( \eta (\mathfrak{e}(\eta)^{-1/2} + \epsilon) \)-spectral set for } \vec{G}_N
        \Big] \asto 1.
    \end{equation*}
    Since this holds for any $\eta>1$, we can take $\eta$ as the value minimizing $\eta \mathfrak{e}(\eta)^{-1/2}(\sigma \eta)^k$.
    Applying \cref{eqn:GMRES_Kss,eqn:minimax} and proves the first part of the statement.

    Next, note that in the proof of \cref{thm:crouzeix}, we in fact establish that, for any \( \eta > 0 \) and \( \epsilon > 0 \), \( \mathcal{D}(0,\sqrt{2}-\eta) \) is a \( (2+\epsilon) \)-spectral set for \( \vec{G}_N \) almost surely as \( N\to\infty \).
    This implies that for any \( \epsilon > 0 \), \( \mathcal{D}(0,\sqrt{2}) \) is a \( (2+\epsilon) \)-spectral set for \( \vec{G}_N \) almost surely as \( N\to\infty \), see \eqref{eq:subset}.
    Applying \cref{eqn:GMRES_Kss,eqn:minimax} and proves the second part of the statement. 
\end{proof}

\subsection{Proof of resolvent norm bound}
\label{sec:resolvent_lip}
In proving \cref{thm:resolvent_bound} we have used the following:

\begin{lemma}
Fix \( z \in \mathbb{C} \) and suppose 
\( \| R(z,\vec{G}) \| < e \), 
\( |z|<f \), 
and \( \| \vec{G} \| < g \).
    Set \( h := \min\{1, (4e^2(f+g+1))^{-1} \} \).
Then for all \( \tilde{z} \) with \( |z- \tilde{z}| < h \)
\begin{equation*}
    | \| R(z, \vec{G}) \| - \| R(\tilde{z}, \vec{G}) \| | < L |z-\tilde z| .
\end{equation*}
where \( L := 3 e^{3} (f+g+1) \).

\end{lemma}

\begin{proof}
Define 
\( \vec{Y}^z := (z -\vec{G})^\cT (z - \vec{G}) \) and \( \vec{Y}^{\tilde z} := (\tilde{z} -\vec{G})^\cT (\tilde{z} - \vec{G}) \).
By a standard eigenvalue bound we have that
\begin{equation*}
    | \lmin( \vec{Y}^z ) - \lmin(\vec{Y}^{\tilde z} ) |
    \leq \| \vec{Y}^z - \vec{Y}^{\tilde{z}} \|.
\end{equation*}
Since \( |z|<f \) and \( |z-\tilde z|< 1 \), then  \( |\tilde{z}|<|z| + |z-\tilde{z}|<f+1 \) so bounding the derivative of $x\mapsto x^2$ gives that
\begin{equation*}
    \big| |z|^2 - |\tilde{z}|^2 \big|
    \leq 2 \max\{|z|,|\tilde z| \} | z-\tilde z|
    \leq 2 (f+1) | z-\tilde z|.
\end{equation*}
    Thus, using the definitions of \( \vec{Y}^z \) and \( \vec{Y}^{\tilde z} \), and the assumption \( \| \vec{G} \| < g \),
\begin{align*}
    \| \vec{Y}^z - \vec{Y}^{\tilde{z}} \| 
    &= \| |z|^2\vec{I} - |\tilde{z}|^2\vec{I} - (z-\tilde{z}) \vec{G}^\cT - (\overline{z-\tilde{z}}) \vec{G} \| 
    \\&\leq \big| |z|^2 - |\tilde{z}|^2 \big| + 2 |z-\tilde{z}| \| \vec{G} \|
    \\&\leq 2(f+g+1) |z-\tilde{z}| .
\end{align*}
Next, note that \( \| \vec{R}(z, \vec{G)} \| < e \) implies \( \lmin(\vec{Y}^z) > e^{-2} \) so using that \( |z-\tilde z| < h \leq (4e^2(f+g+1))^{-1} \), 
\begin{equation*}
    \lmin(\vec{Y}^{\tilde z}) > e^{-2} - 2(f+g+1) h > e^{-2}/2.
\end{equation*}
Next, note that bounding the derivative of \( x\mapsto x^{-1/2} \) gives that for any \( x, y > 0 \),
\begin{equation*}
    | x^{-1/2} - y^{-1/2} | \leq  \frac{1}{2}\min\{x,y\}^{-3/2} |x-y|.
\end{equation*}
Thus, 
\begin{equation*}
    | \lmin(\vec{Y}^{z})^{-1/2} - \lmin(\vec{Y}^{\tilde z})^{-1/2} |
    < \frac{1}{2}(e^{-2}/2)^{-3/2} (2(f+g+1))|z-\tilde{z}|
\end{equation*}
    The result follows from the definition of \( L \) and the fact that \( 2^{3/2} < 3 \).
\end{proof}

\subsection{Other proofs}

\begin{proof}[Proof of \cref{thm:hausdorff_disks}]
    Suppose \( \mathcal{D}(0,r-\epsilon) \subseteq S \subseteq \mathcal{D}(0,r+\epsilon)  \).
    Let \( x \in \mathcal{D}(0,r) \). 
    Then \( x \) is within \( \epsilon \) of a point in \( \mathcal{D}(0,r-\epsilon) \) and therefore of \( S \).
    Now, let \( x\in S \). 
    Then \( |x| \leq r+\epsilon \) so \( x \) is within \( \epsilon \) of a point in \( \mathcal{D}(0,r) \).

    Certainly \( S \subseteq \mathcal{D}(0,r+\epsilon) \).
    Suppose now that \( S \) is convex and that \( d_H(S,\mathcal{D}(0,r)) < \epsilon \).
    Let \( x \in \mathcal{D}(0,r-\epsilon) \). 
    Let \( y \) be the point on the boundary of \( S \) with same argument as \( x \) and, for the sake of contradiction, that \( |y|<|x| \).
    Let \( z \) be the point on the boundary of \( \mathcal{D}(0,r) \) such that the segment between \( y \) and \( z \) is perpendicular to some line which passes through \( y \) but no other points in \( S \). 
    Now, note that this line separates \( z \) from \( S \) and that all points in \( S \) are therefore a distance at least \( |z-y| > \epsilon \) from \( z \), contradicting the assumption \( d_H(S,\mathcal{D}(0,r)) < \epsilon \).
    Thus, we find that \( |y| \geq |x| \) which implies that \( \mathcal{D}(0,r-\epsilon) \subset S \).
\end{proof}

\begin{proof}[Proof of \cref{eqn:as_cond}]
Suppose $A_N\subset B_N$ and $\bOne[A_N]$ converges almost surely to 1 as $N\to\infty$.
Then, for all \( N \) and for all \( \omega \in \Omega \), \( \omega \in A_N \Longrightarrow \omega \in B_N  \) so
\begin{equation*}
    \bOne[A_N](\omega) = 1 \quad\Longrightarrow\quad \bOne[B_N](\omega)  = 1,
    \qquad 
    \forall N,\omega.
\end{equation*}
Therefore, since $\bOne[C](\omega)$ is either zero or one for any event $C$,
\begin{equation*}
    \PP \left[ \lim_{N\to\infty} \bOne[B_N] = 1 \right] 
    \geq \PP \left[ \lim_{N\to\infty} \bOne[A_N] = 1 \right] 
    = 1,
\end{equation*}
so \( \bOne[B_N] \asto 1 \).
\end{proof}

\appendix

\section{\nopunct Is the field of values of a large Ginibre matrix a \( \sqrt{2} \)-spectral set?}
\label{sec:nr_blaschke}
As we mentioned above, the function \( f \) which which attains the ratio \( C(W(\vec{A}),\vec{A}) \) is of the form \( f = B \circ \phi \) where \( B \) is a finite Blaschke product \( B(z) = \exp(\ii\theta) \prod_{i=1}^{N-1} ({z-\alpha_i})/({1-\overline{\alpha_i} z}) \) of degree at most \( N-1 \) and \( \phi \) is any conformal mapping from the field of values to the unit disk.
Here we assume \( | \alpha_i| \leq 1  \) to allow products with degree less than \( N-1 \).
Numerical codes to search for extremal Blaschke products have been used in previous studies of Crouzeix's conjecture \cite{bickelgorkin_greenbaum_ransford_schwenninger_wegert_20}.

\begin{figure}
    \centering
    \includegraphics[scale=0.8]{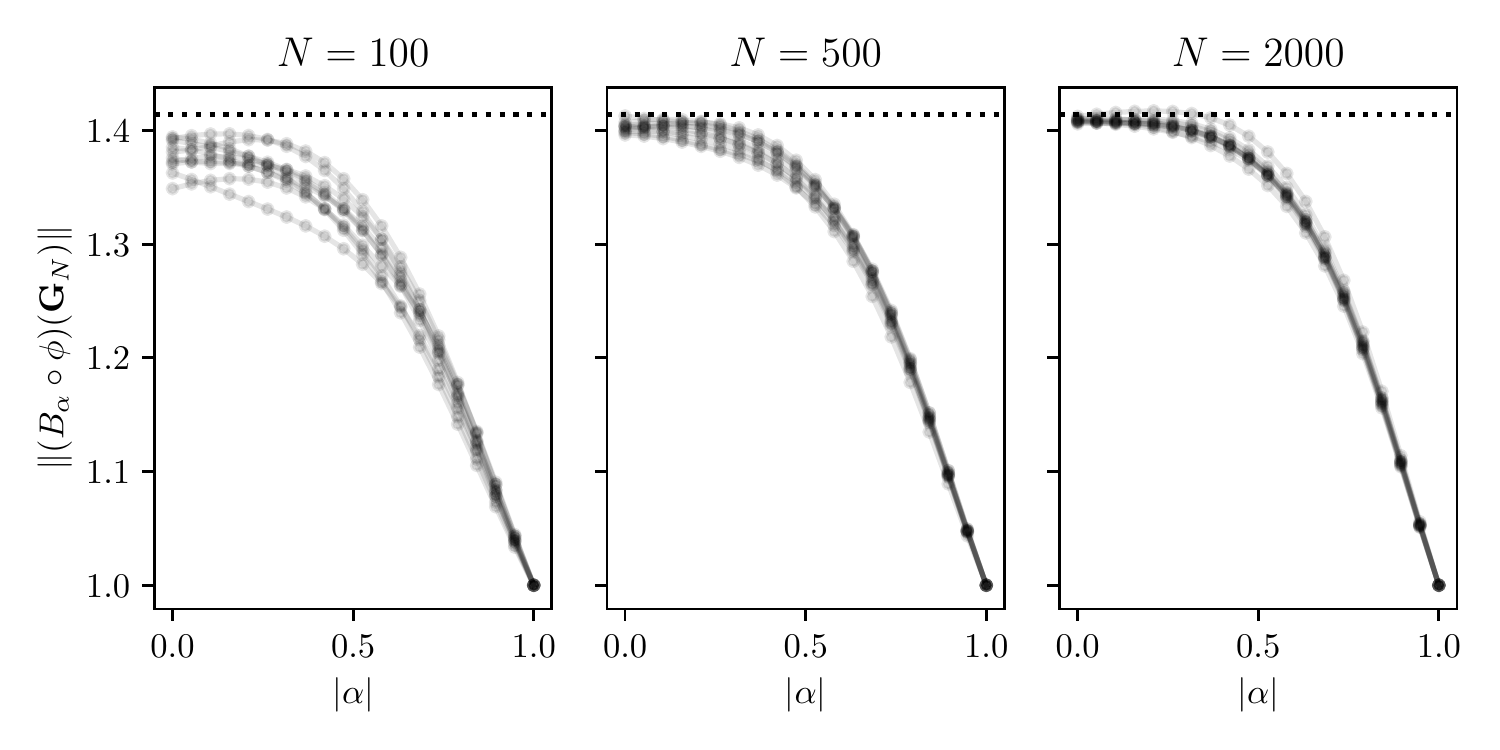}
    \caption{Relationship between \( \| (B_{\alpha}\circ \phi)(\vec{G}_N) \| \) and \( \alpha \) for 10 samples of \( \vec{G}_N \) when \( \phi(z) = z/\sqrt{2} \).
    The dotted line is at \( \sqrt{2} \).
    Here we take \( \alpha \) on the real axis, although the same result is expected for any \( |\alpha|<1 \) due the rotational symmetry of \( \vec{G}_N \).
    }
    \label{fig:BP_norm}
\end{figure}

For large Ginibre matrices, \( W(\vec{G}_N) \approx \mathcal{D}(0,\sqrt{2}) \). Therefore, we expect that \( C(W(\vec{G}_N),\vec{G}_N) \approx  C(\mathcal{D}(0,\sqrt{2}),\vec{G}_N)  \).
To approximate \( C(\mathcal{D}(0,\sqrt{2}),\vec{G}_N) \) we first approximately scale the problem to the unit disk using \( \phi(z) = z/\sqrt{2} \) and then apply a black box numerical optimzer to try and find the roots \( \{\alpha_i\} \) of the Blaschke product maximizing \( \| (B\circ\phi)(\vec{G}_N) \| \).
We observe that the \( B \) we compute are \emph{numerically} close to degree one Blaschke products.
That is, if we allow for higher degree products, we find that all but one of the \( \{ \alpha_i \} \) have magnitude extremely close to 1 and contribute very little to the norm of the resulting product.

Thus, in the \( N\to\infty \) limit, we might expect that \( f = B_{\alpha} \circ \phi \) where \( B_{\alpha}(z) = (z-\alpha)/(1-\overline{\alpha} z) \) for some value of \( \alpha \) and \( \phi(z) = z/\sqrt{2} \).
In \cref{fig:BP_norm}, we explore the relationship between \( \| (B_{\alpha}\circ \phi) (\vec{G}_N) \| \) and \( \alpha \) when \( \phi(z) = z/\sqrt{2} \).
We observe that this quantity seems to concentrate about some curve which is bounded above by \( \sqrt{2} \) and maximized at \( \alpha = 0 \).
In the \( \alpha = 0 \) case we have that \( (B_{0} \circ \phi ) (\vec{G}_N) = \vec{G}_N / \sqrt{2} \) so that \( \| (B_{0}\circ \phi)(\vec{G}_N) \| \asto \sqrt{2} \).

\bibliography{random_GMRES_Crouzeix}
\bibliographystyle{siam}

\end{document}